\documentclass{proc-l}

\usepackage{mathptmx}       
\usepackage{helvet}         
\usepackage{courier}        
\usepackage{type1cm}

\usepackage{graphicx}
\usepackage[bottom]{footmisc}

\usepackage{amsthm, amssymb, amsfonts}
\usepackage{MnSymbol}
\usepackage{tikz}
\usetikzlibrary{automata,positioning}
\usepackage{float}
\usepackage{array}
\usepackage{listofitems}
\usepackage{hyperref}
\usepackage{textcomp}
\usepackage[thinlines]{easytable}

\newtheorem{theorem}{Theorem}[section]
\newtheorem{lemma}[theorem]{Lemma}
\newtheorem{proposition}[theorem]{Proposition}

\theoremstyle{definition}

\newcommand{\fifteenpuzzle}[1]{
\parbox[c]{3cm}{\centering
\begin{TAB}(e,0.5cm,0.5cm){|c|c|c|c|}{|c|c|c|c|}
#1
\end{TAB}}}

\newcommand{\varikon}[8]{
\parbox[c]{3cm}{\centering
\begin{tikzpicture}[every node/.style={minimum size=1.0cm},on grid]
\begin{scope}[every node/.append style={yslant=-0.5},yslant=-0.5]
  \node at (0.25,0.75) {#3};
  \node at (0.75,0.75) {#4};
  \node at (0.25,0.25) {#7};
  \node at (0.75,0.25) {#8};
  \draw[step=0.5] (0,0) grid (1,1);
\end{scope}
\begin{scope}[every node/.append style={yslant=0.5},yslant=0.5]
  \node at (1.25,-0.25) {#4};
  \node at (1.75,-0.25) {#2};
  \node at (1.25,-0.75) {#8};
  \node at (1.75,-0.75) {#6};
  \draw[step=0.5] (1,-1) grid (2,0);
\end{scope}
\begin{scope}[every node/.append    style={yslant=-0.5,xslant=1},yslant=-0.5,xslant=1]
  \node at (-0.75,1.75) {#1};
  \node at (-0.25,1.75) {#2};
  \node at (-0.75,1.25) {#3};
  \node at (-0.25,1.25) {#4};
  \draw[step=0.5] (-1,1) grid (0,2);
\end{scope}
\end{tikzpicture}}}

\newcommand{\cyclesix}[6]{
\parbox[c]{2.5cm}{\centering
\begin{tikzpicture}[shorten >=1pt,node distance=1cm,on grid,auto]
    \node[] (center) {}; 
    \node[] (1) [above right = {1*sin(90)} and {1*(cos(90)} of center]  {#1};
    \node[] (2) [above right = {1*sin(150)} and {1*(cos(150)} of center]  {#2};
    \node[] (3) [above right = {1*sin(210)} and {1*(cos(210)} of center]  {#3};
    \node[] (4) [above right = {1*sin(270)} and {1*(cos(270)} of center]  {#4};
    \node[] (5) [above right = {1*sin(330)} and {1*(cos(330)} of center]  {#5};
    \node[] (6) [above right = {1*sin(390)} and {1*(cos(390)} of center]  {#6};
    \path[->] 
        (1) edge node {} (2)
        (1) edge node {} (3)
        (1) edge node {} (5)
        (1) edge node {} (6)
        
        (3) edge node {} (1)
        (3) edge node {} (2)
        (3) edge node {} (4)
        (3) edge node {} (5)
        
        (5) edge node {} (1)
        (5) edge node {} (3)
        (5) edge node {} (4)
        (5) edge node {} (6)
        
        (2) edge node {} (1)
        (2) edge node {} (3)
        
        (4) edge node {} (3)
        (4) edge node {} (5)
        
        (6) edge node {} (1)
        (6) edge node {} (5);
\end{tikzpicture}}}

\newcommand{\cyclefive}[5]{
\parbox[c]{2.5cm}{\centering
\begin{tikzpicture}[shorten >=1pt,node distance=1cm,on grid,auto]
    \node[] (center) {}; 
    \node[] (1) [above right = {1*sin(90)} and {1*(cos(90)} of center]  {#1};
    \node[] (2) [above right = {1*sin(150)} and {1*(cos(150)} of center]  {#2};
    \node[] (3) [above right = {1*sin(210)} and {1*(cos(210)} of center]  {#3};
    \node[] (4) [above right = {1*sin(270)} and {1*(cos(270)} of center]  {#4};
    \node[] (5) [above right = {1*sin(330)} and {1*(cos(330)} of center]  {#5};
    \path[->] 
        (1) edge node {} (2)
        (1) edge node {} (3)
        
        (3) edge node {} (1)
        (3) edge node {} (2)
        (3) edge node {} (4)
        (3) edge node {} (5)
        
        (5) edge node {} (3)
        (5) edge node {} (4)
        
        (2) edge node {} (1)
        (2) edge node {} (3)
        
        (4) edge node {} (3)
        (4) edge node {} (5);
\end{tikzpicture}}}

\newcommand\cycle[2][\,]{%
  \readlist\thecycle{#2}%
  (\foreachitem\i\in\thecycle{\ifnum\icnt=1\else#1\fi\i})%
}

\begin{document}

\title{Algebraic Structure of the Varikon Box}

\author{Jason d'Eon}
\address{Dalhousie University, 6299 South Street, Halifax, Nova Scotia, Canada}
\email{js348697@dal.ca}

\author{Chrystopher L. Nehaniv}
\address{University of Waterloo, 200 University Avenue West, Waterloo, Ontario, Canada}
\email{chrystopher.nehaniv@uwaterloo.ca}

\date{June 1, 2020}

\begin{abstract}
The 15-Puzzle is a well studied permutation puzzle. This paper explores the group structure of a three-dimensional variant of the 15-Puzzle known as the Varikon Box, with the goal of providing a heuristic that would help a human solve it while minimizing the number of moves. First, we show by a parity argument which configurations of the puzzle are reachable. We define a generating set based on the three dimensions of movement, which generates a group that acts on the puzzle configurations, and we explore the structure of this group. Finally, we show a heuristic for solving the puzzle by writing an element of the symmetry group as a word in terms of a generating set, and we compute the shortest possible word for each puzzle configuration.
\end{abstract}

\maketitle

\section{Introduction}
\label{sec:introduction}
The 15-Puzzle is a permutation puzzle which consists of a 4$\times$4 grid with fifteen numbered squares and one empty space that allows the pieces to slide around. Many variants of the 15-Puzzle exist, for example, by changing the size of the grid. Aside from the 15-Puzzle, one could consider the 24-Puzzle, 8-Puzzle, or the very trivial 3-Puzzle, which correspond to a 5$\times$5 grid, a 3$\times$3 grid, and a 2$\times$2 grid respectively. In fact, there is no particular reason the board has to be square, so one could take the puzzle consisting of a 2$\times$3 grid, with five movable pieces.

The focus of this paper is on a three-dimensional permutation puzzle known as the 2$\times$2$\times$2 Varikon Box. It also consists of a 2$\times$2$\times$2 grid with seven movable pieces. Each piece has three sides coloured red and three sides coloured blue, so that opposite faces are opposite colours (shown in Figure~\ref{fig:cubie}). There is always one corner of the piece surrounded by red faces and one corner surrounded by blue faces, giving eight distinct orientations of a piece: one for each position of the ``blue'' corner (which also determines the position of the red corner). The seven pieces in the puzzle have distinct orientations. A solved state of the puzzle is a configuration where all the faces towards the outside of the puzzle are one colour, but the three faces in the core, seen through the empty space, are the opposite colour.

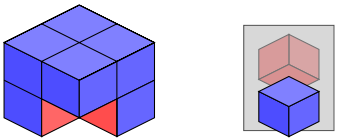
\begin{figure}
    \centering{
    \begin{tikzpicture}[every node/.style={minimum size=1.0cm},on grid]
    \begin{scope}[every node/.append style={yslant=-0.5},yslant=-0.5]
        \draw[fill=blue!70!white] (0,0) rectangle (0.5,1);
        \draw[fill=blue!70!white] (0.5,0.5) rectangle (1,1);
        \draw[fill=red!60!white] (0.5,0) -- (0.5,0.5) -- (1,0.5) -- cycle;
        \draw (0,0.5) -- (0.5,0.5);
    \end{scope}
    \begin{scope}[every node/.append style={yslant=0.5},yslant=0.5]
        \draw[fill=blue!60!white] (1,0) rectangle (2,-0.5);
        \draw[fill=blue!60!white] (1.5,-0.5) rectangle (2,-1);
        \draw[fill=red!70!white] (1,-0.5) -- (1.5,-0.5) -- (1.5,-1) -- cycle;
        \draw (1.5,0) -- (1.5,-1);
    \end{scope}
    \begin{scope}[every node/.append style={yslant=-0.5,xslant=1},yslant=-0.5,xslant=1]
        \draw[step=0.5, fill=blue!50!white] (-1,1) grid (0,2) rectangle (-1,1);
    \end{scope}
    \end{tikzpicture}
    \hspace{1cm}
    \begin{tikzpicture}[scale=0.4, every node/.style={minimum size=1.0cm}]
    \begin{scope}[every node/.append style={yslant=-0.5},yslant=-0.5]
        \filldraw[draw=black,fill=red!50!white,step=0.5] (1,2.2) rectangle (2,3.2);
    \end{scope}
    \begin{scope}[every node/.append style={yslant=0.5},yslant=0.5]
        \filldraw[draw=black,fill=red!60!white,step=0.5] (0,1.2) rectangle (1,2.2);
    \end{scope}
    \begin{scope}[every node/.append    style={yslant=-0.5,xslant=1},yslant=-0.5,xslant=1]
        \filldraw[draw=black,fill=red!70!white,step=0.5] (-1.2,1.2) rectangle (-0.2,2.2);
    \end{scope}
        
    \filldraw[draw=black,fill=lightgray,opacity=0.6] (-0.5,-0.5) rectangle (2.5,3);
        
    \begin{scope}[every node/.append style={yslant=-0.5},yslant=-0.5]
        \filldraw[draw=black,fill=blue!70!white,step=0.5] (0,-0.2) rectangle (1,0.8);
    \end{scope}
    \begin{scope}[every node/.append style={yslant=0.5},yslant=0.5]
        \filldraw[draw=black,fill=blue!60!white,step=0.5] (1,-1.2) rectangle (2,-0.2);
    \end{scope}
    \begin{scope}[every node/.append    style={yslant=-0.5,xslant=1},yslant=-0.5,xslant=1]
        \filldraw[draw=black,fill=blue!50!white,step=0.5] (-0.8,0.8) rectangle (0.2,1.8);
    \end{scope}
    
    \end{tikzpicture}}
    \caption{The Varikon Box. The left shows an example of a solved configuration. The right shows two views of a piece inside the 2$\times$2$\times$2 Varikon Box. It has one corner surrounded by blue faces, and the opposite corner surrounded by red faces.}
    \label{fig:cubie}
\end{figure}

The outline of the paper is as follows. Section \ref{sec:fifteen} is a review of the classical analysis of 15-Puzzle configurations which can be reached by valid moves. Section \ref{sec:varikon} covers which properties of the 15-Puzzle carry over to the 2$\times$2$\times$2 Varikon Box, and describes its reachable configurations. Section \ref{sec:group_action} describes the structure of the group formed by the moves of the 2$\times$2$\times$2 Varikon Box. Section \ref{sec:shortest_word} gives a heuristic for solving the Varikon Box in few moves, by writing permutations as words in terms of a generating set. Section \ref{sec:conclusion} concludes with some open questions on generalizations of these puzzles.

\section{Review of the 15-Puzzle}
\label{sec:fifteen}

The 2$\times$2$\times$2 Varikon Box is closely related with the 15-Puzzle, so we begin by reviewing the structure of the 15-Puzzle. In particular, we are interested in configurations which we can reach using a valid sequence of moves. Let us denote by $C$ the set of such reachable configurations of the 15-Puzzle. For convenience, we denote the solved configuration by $\iota$. A sequence of valid moves permutes the pieces of the puzzle, but it is not the case that every permutation of the pieces is reachable. For example, Figure~\ref{fig:not_reachable} shows a configuration which is well-known not to be in $C$.

\begin{figure}[H]
    \centering{
        \fifteenpuzzle{
        1 & 2 & 3 & 4 \\
        5 & 6 & 7 & 8 \\
        9 & 10 & 11 & 12 \\
        13 & 15 & 14 &
        }
    }
    \caption{An unsolvable configuration of the 15-Puzzle. There does not exist a sequence of moves that maps this configuration to the solved state.}
    \label{fig:not_reachable}
\end{figure}

Now take the subset $C_{fix}$ of reachable configurations where the empty space is fixed in the bottom right corner. Configurations in $C_{fix}$ can also be thought of as permutations in $S_{15}$, with respect to $\iota$. For example, the configuration in Figure~\ref{fig:not_reachable} corresponds with the permutation $\cycle[\;]{14,15}$, since performing this permutation on the pieces of $\iota$ would yield the configuration in the figure.  The following lemma was first shown in~\cite{Johnson}.

\begin{lemma}\label{lem:fifteen_fix_even}
For every $c \in C_{fix}$, $c$ must correspond with an even permutation.
\end{lemma}

\begin{proof}
This can be seen by imagining the 4x4 grid as a black and white checkerboard. When considering moves that swap the empty space with an adjacent square, each move must change the colour that the empty space is on. If we take two configurations $c_1, c_2 \in C_{fix}$, it must take an even number of transpositions to transition from $c_1$ to $c_2$, since the empty space begins and ends on the same colour.

\end{proof}

To show that every even permutation is a reachable configuration, we adapt the proof from~\cite{Mulholland}. First, we introduce a notation for moves, which act as maps on the configurations. The definition for moves is based on the idea of sliding blocks to the right or left, as well as up or down. Unfortunately, not all moves are possible on all configurations. If the empty space is on the far left side of the grid, there is no piece to the right which can be moved to fill the space. To fix this issue, we define a ``right'' move, denoted $R$, to either mean sliding a block to the right to fill the space, or if the space is on the far left, it means to slide the entire row to the left. Figure~\ref{fig:r_move_order} shows that under this definition, $R^4$ is equivalent to the identity map on $C$, and $R^3$ is what we might consider a ``left'' move. Similarly, we can define $U$ to be the ``up'' move, which slides a piece up, effectively moving the empty space down. Using this notation, moves can be written as a sequence of $R$'s and $U$'s.

\begin{figure}
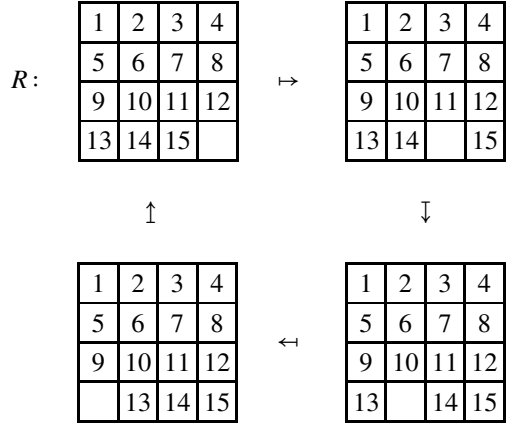

    \centering
    $R:$\fifteenpuzzle{
        1 & 2 & 3 & 4 \\
        5 & 6 & 7 & 8 \\
        9 & 10 & 11 & 12 \\
        13 & 14 & 15 & 
    }
    $\mapsto$
    \fifteenpuzzle{
        1 & 2 & 3 & 4 \\
        5 & 6 & 7 & 8 \\
        9 & 10 & 11 & 12 \\
        13 & 14 & & 15
    }\\
    \vspace{0.5cm}
    $\quad\upmapsto\hspace{3.5cm}\downmapsto$\\
    \vspace{0.5cm}
    \quad\fifteenpuzzle{
        1 & 2 & 3 & 4 \\
        5 & 6 & 7 & 8 \\
        9 & 10 & 11 & 12 \\
         & 13 & 14 & 15
    }
    $\leftmapsto$
    \fifteenpuzzle{
        1 & 2 & 3 & 4 \\
        5 & 6 & 7 & 8 \\
        9 & 10 & 11 & 12 \\
        13 & & 14 & 15
    }
    \caption{Applying the ``right'' move repeatedly cycles through four 15-Puzzle configurations.}
    \label{fig:r_move_order}
\end{figure}

The diagram on the left in Figure~\ref{fig:fifteen_moves} is in $C_{fix}$, as it can be obtained by applying $R U^3 R^3 U$ to $\iota$, which is indicated by multiplication. Using the convention that moves are applied from left to right, this sequence corresponds with the permutation $\cycle[\;]{11,12,15}$. In order to show that every even permutation is reachable, we will use the fact that $A_{15}$ is generated by the $3$-cycles, $\cycle[\;]{11,12,i}$ for all $i \in \{1, \dots, 15\}$ other than $11$ and $12$.

\begin{figure}[H]
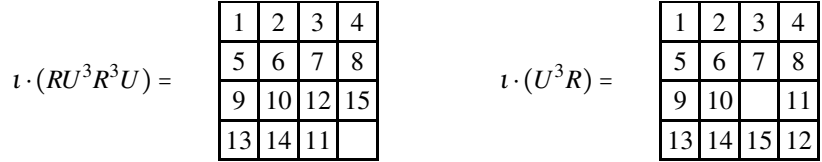

    \centering
    $\iota \cdot (R U^3 R^3 U) = $ \fifteenpuzzle{
        1 & 2 & 3 & 4 \\
        5 & 6 & 7 & 8 \\
        9 & 10 & 12 & 15 \\
        13 & 14 & 11 & 
    }
    \hspace{1cm}
    $\iota \cdot (U^3 R) = $ \fifteenpuzzle{
        1 & 2 & 3 & 4 \\
        5 & 6 & 7 & 8 \\
        9 & 10 &  & 11 \\
        13 & 14 & 15 & 12
    }
    \caption{Examples of sequences being applied to the solved configuration of the 15-Puzzle. The left shows a $3$-cycle in the bottom-right quadrant, and the right shows the set-up sequence, $U^3 R$, being applied to $\iota$.}
    \label{fig:fifteen_moves}
\end{figure}

We start by performing a set-up sequence, $U^3 R$, to the solved configuration, which we will undo later (Figure~\ref{fig:fifteen_moves}). Following the set-up, we can then swap the empty space with the following sequence of pieces: $7 \to 8 \to 4 \to 3 \to 2 \to 1 \to 5 \to 6 \to 10 \to 9 \to 13 \to 14 \to 15 \to 7$. Repetitions of this cycle will replace $15$ in the bottom-right quadrant with any other piece in this sequence, while fixing $11$ and $12$. Altogether, this is written:
\begin{equation}
    \sigma_n = (U^3 R) (U^3 R^3 U^3 R^3 U R^3 U R U R^2 U^3)^n (U^3 R)^{-1},
\end{equation}
where $n \geq 0$. Therefore, over all distinct choices of $n$, the sequence:
\begin{equation}
    \sigma_n (R U^3 R^3 U) \sigma_n^{-1},
\end{equation}
will correspond with permutations of the form $\cycle[\;]{11,12,i}$, for all $i$ except $i=11,12$. By the above arguments, we have the following theorem.

\begin{theorem}\label{thm:fifteen}
$C_{fix}$ corresponds precisely with even permutations of the fifteen pieces.
\end{theorem}

A similar argument applies to any fixed position of the empty space. Therefore, the number of reachable configurations is $16 \cdot |A_{15}| = \frac{16!}{2}$. This is different than saying the reachable configurations are even permutations of the $16$ squares. Rather, when the empty square is an even (or respectively, odd) number of swaps away from the bottom-right corner, then the configuration is reachable if and only if the permutation is even (respectively odd).

\section{Reachable Configurations of the Varikon Box}
\label{sec:varikon}

In this section, we describe some previously known results for the 2$\times$2$\times$2 Varikon Box~\cite{Scherphuis}, and provide proofs for these facts, by generalizing the 15-Puzzle. First, we note that the 2$\times$2$\times$2 Varikon Box is precisely a three-dimensional variant of the 15-Puzzle. If we choose to try solving the puzzle by making all the outer faces blue, there would be only one candidate solution, which consists of matching the blue corner of each piece with the respective corner of the puzzle. Therefore, at first glance, there would appear to be two solutions: whether we choose to put blue or red on the outer faces.

\begin{lemma}\label{lem:one_soln}
Given a fixed starting configuration, the 2$\times$2$\times$2 Varikon Box has exactly one solution.
\end{lemma}
\begin{proof}
On each individual piece, the red corner and blue corner must be opposite from each other. Transitioning between candidate solutions would mean swapping every piece with the contents of the opposite corner, which is an even permutation. However, any sequence of moves that takes the empty space to the opposite corner will involve an odd number of swaps. Therefore, it is not possible to transition between the two candidate solutions, making only one possible to reach.
\end{proof}

According to Lemma \ref{lem:one_soln}, we can numerically label the pieces and define a solved configuration in terms of the labelling. Let us denote the set of reachable configurations of the 2$\times$2$\times$2 Varikon Box by $V$. We will reuse $\iota$ to indicate the solved configuration, shown in Figure~\ref{fig:varikon_solved}.

\begin{figure}[H]
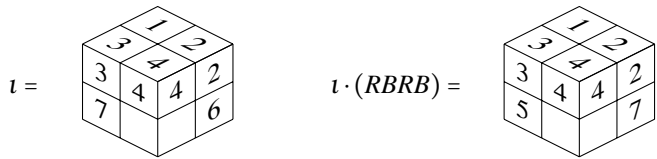

    \centering
    $\iota = $\varikon{1}{2}{3}{4}{5}{6}{7}{}
    \qquad
    $\iota \cdot (RBRB) = \varikon{1}{2}{3}{4}{6}{7}{5}{}$
    \caption{Left: by labelling the pieces of the 2$\times$2$\times$2 Varikon Box, we can define the labelling that represents the solved configuration, denoted by $\iota$. Right: an example of a 3-cycle on the bottom half of the 2$\times$2$\times$2 Varikon Box.}
    \label{fig:varikon_solved}
\end{figure}

Lemma \ref{lem:fifteen_fix_even} extends to the three-dimensional case. Let $V_{fix}$ be the set of configurations where the empty space is in its solved position. If $v \in V_{fix}$, it must correspond with an even permutation of the $7$ pieces: a fact which we already used in the proof of Lemma \ref{lem:one_soln}. To prove that every even permutation is in $V_{fix}$, we show that every cycle of the form $\cycle[\;]{5,6,i}$ is in $V_{fix}$, when $i \neq 5,6$, as this will generate $A_7$.

To describe sequences of swaps on the Varikon Box, we need three generators: $R,U,B$ (right, up, and back, respectively), which act as according to Figure~\ref{fig:varikon_moves}. We define $R^2, U^2,$ and $B^2$ to be the identity map to fix the issue of certain moves being impossible given the position of the empty space. To get the permutation $\cycle[\;]{5,6,7}$, one can perform sequence $RBRB$ (Figure~\ref{fig:varikon_solved}), but we can also replace the $7$ with any of the other pieces, by swapping the empty space with $4 \to 2 \to 1 \to 3 \to 7 \to 4$. By repeating the cycle, we can replace $7$ with any other piece, in order to perform $\cycle[\;]{5,6,i}$ for other $i$.

\begin{figure}[H]
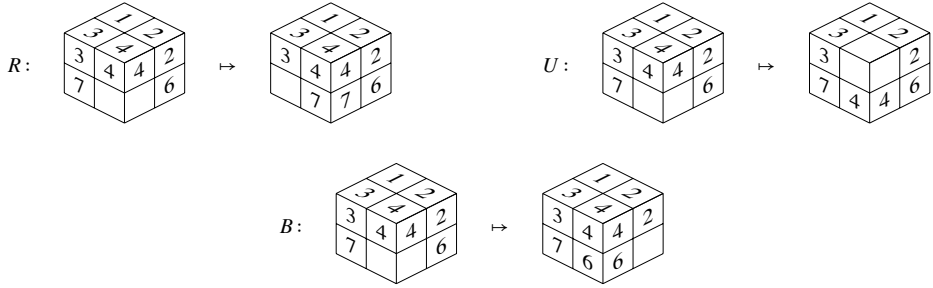

    \centering
    \scalebox{0.8}{$R: \varikon{1}{2}{3}{4}{5}{6}{7}{} \mapsto \varikon{1}{2}{3}{4}{5}{6}{}{7}$}\hfill
    \scalebox{0.8}{$U: \varikon{1}{2}{3}{4}{5}{6}{7}{} \mapsto \varikon{1}{2}{3}{}{5}{6}{7}{4}$}
    \scalebox{0.8}{
    $B: \varikon{1}{2}{3}{4}{5}{6}{7}{} \mapsto \varikon{1}{2}{3}{4}{5}{}{7}{6}$}
    \caption{The three possible directions of movement for the 2$\times$2$\times$2 Varikon Box.}
    \label{fig:varikon_moves}
\end{figure}

Therefore, we can extend Theorem \ref{thm:fifteen} to the 2$\times$2$\times$2 Varikon Box, since configurations in $V_{fix}$ correspond with even permutations of the $7$ pieces. By symmetry, we can conclude that there are $\frac{8!}{2} = 20,160$ reachable configurations, since for every position of the empty space, we can perform any even permutation on the $7$ pieces.

\section{Group Structure of the 2$\times$2$\times$2 Varikon Box}
\label{sec:group_action}

We begin this section by observing that sequences of $R, U,$ and $B$ give rise to a group-like structure.

\begin{proposition}\label{prop:var_group}
Let $s_1,s_2$ be two sequences of $R, U,$ and $B$. We say $s_1 = s_2$ if for all $c_1, c_2 \in V$, $s_1: c_1 \mapsto c_2$ if and only if $s_2: c_1 \mapsto c_2$. Then with respect to composition, the set of sequences form a group and the mapping on the configurations is equivalent to a group action.
\end{proposition}

\begin{proof}
Composition is associative and concatenating two sequences will produce another valid sequence. The empty sequence satisfies the properties of the identity. Every sequence is invertible, since each element of the generating set $\{R, U, B\}$ is an involution. The group operation is well-defined, since if $x$ and $y$ are sequences where $x=y$ and $s$ another sequence, then $xs = ys$, since for all $c \in V$, $x$ and $y$ map $c$ to the same configuration, and performing additional moves will maintain equality. The mapping on configurations is clearly a group action, since the empty sequence leaves all configurations untouched, and the group multiplication is defined to be compatible with the action.
\end{proof}

We now investigate the structure of this group, which we call $G$, and we show how to reduce it to a structure that will help us solve the 2$\times$2$\times$2 Varikon Box. First, note that the stabilizer of $\iota$ is trivial and the action of the group is transitive, which implies that $\vert G \vert = 20,160$.

Interestingly, if we restrict $G$ to the subgroup of sequences involving only $R$ and $U$, we get a copy of $D_6$, since $R^2 = e, (RU)^6 = e$, and $RU \cdot R = R \cdot (RU)^{-1}$. For any given configuration, this gives a local picture around the configuration, since by alternating any two of $R,U,$ and $B$, we obtain a copy of $D_6$, pictured in Figure~\ref{fig:varikon_cycles}.

\begin{figure}[H]
    \centering
    \begin{tikzpicture}[shorten >=1pt,node distance=2cm,on grid,auto,
        state/.style={circle, draw, minimum size=0.5cm}]
    \node[state, initial where=above] (center) {}; 
    \node[above left = 1.3 and 2.3 of center] (RU) {};
    \node[above right = 1.3 and 2.3 of center] (RB) {};
    \node[below = 2.4 of center] (UB) {};
    
    \foreach \i [count=\ni] in {90,210,-30}
        \node[state] (\ni) [above right = {sin(\i)} and {cos(\i)} of center]  {};
    
    \foreach \i [count=\ni] in {30,60,...,270}
        \node[state] (RU\ni) [above right = {2*sin(\i)} and {2*(cos(\i)} of RU]  {};
    
    \foreach \i [count=\ni] in {150,120,...,-90}
        \node[state] (RB\ni) [above right = {2*sin(\i)} and {2*(cos(\i)} of RB]  {};
    
    \foreach \i [count=\ni] in {150,180,...,390}
        \node[state] (UB\ni) [above right = {2*sin(\i)} and {2*(cos(\i)} of UB]  {};
    \path[->]
        (center) edge node {R} (1)
        (center) edge node {U} (2)
        (center) edge node {B} (3)
        (1) edge node {R} (center)
        (2) edge node {U} (center)
        (3) edge node {B} (center)
        
        (1) edge node {U} (RU1)
        (RU1) edge node {R} (RU2)
        (RU2) edge node {U} (RU3)
        (RU3) edge node {R} (RU4)
        (RU4) edge node {U} (RU5)
        (RU5) edge node {R} (RU6)
        (RU6) edge node {U} (RU7)
        (RU7) edge node {R} (RU8)
        (RU8) edge node {U} (RU9)
        (RU9) edge node {R} (2)
        
        (RU1) edge node {} (1)
        (RU2) edge node {} (RU1)
        (RU3) edge node {} (RU2)
        (RU4) edge node {} (RU3)
        (RU5) edge node {} (RU4)
        (RU6) edge node {} (RU5)
        (RU7) edge node {} (RU6)
        (RU8) edge node {} (RU7)
        (RU9) edge node {} (RU8)
        (2) edge node {} (RU9)
        
        (1) edge node {} (RB1)
        (RB1) edge node {} (RB2)
        (RB2) edge node {} (RB3)
        (RB3) edge node {} (RB4)
        (RB4) edge node {} (RB5)
        (RB5) edge node {} (RB6)
        (RB6) edge node {} (RB7)
        (RB7) edge node {} (RB8)
        (RB8) edge node {} (RB9)
        (RB9) edge node {} (3)
        
        (RB1) edge node {B} (1)
        (RB2) edge node {R} (RB1)
        (RB3) edge node {B} (RB2)
        (RB4) edge node {R} (RB3)
        (RB5) edge node {B} (RB4)
        (RB6) edge node {R} (RB5)
        (RB7) edge node {B} (RB6)
        (RB8) edge node {R} (RB7)
        (RB9) edge node {B} (RB8)
        (3) edge node {R} (RB9)
        
        (2) edge node {B} (UB1)
        (UB1) edge node {U} (UB2)
        (UB2) edge node {B} (UB3)
        (UB3) edge node {U} (UB4)
        (UB4) edge node {B} (UB5)
        (UB5) edge node {U} (UB6)
        (UB6) edge node {B} (UB7)
        (UB7) edge node {U} (UB8)
        (UB8) edge node {B} (UB9)
        (UB9) edge node {U} (3)
        
        (UB1) edge node {} (2)
        (UB2) edge node {} (UB1)
        (UB3) edge node {} (UB2)
        (UB4) edge node {} (UB3)
        (UB5) edge node {} (UB4)
        (UB6) edge node {} (UB5)
        (UB7) edge node {} (UB6)
        (UB8) edge node {} (UB7)
        (UB9) edge node {} (UB8)
        (3) edge node {} (UB9);
    \end{tikzpicture}
    \caption{Centered at a particular configuration, if one alternates between $R$ and $U$, between $R$ and $B$, or between $U$ and $B$, we get three copies of the dihedral group of order $12$.}
    \label{fig:varikon_cycles}
\end{figure}
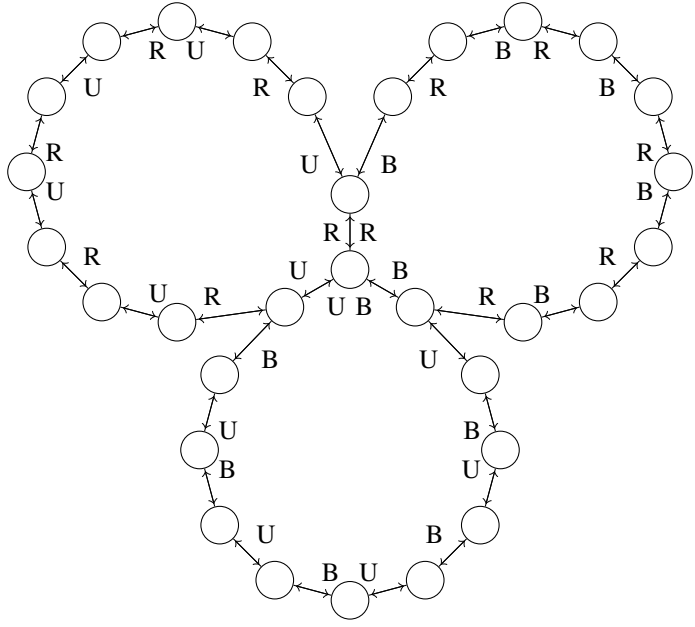

To help break down the size of the group, 
consider the group homomorphism, $\varphi:G \to (\mathbb{Z}_2)^3$, where for $g \in G$ the components of $\varphi(g)$ correspond with the counts modulo $2$ of $R$'s, $U$'s, and $B$'s in $g$ respectively. For example:
\begin{equation}
    \varphi(RUBUBR) = (0,0,0),
\end{equation}
which also implies that $RUBUBR$ fixes the empty space. It is clear to see that this is a well-defined group homomorphism by properties of modular arithmetic since each letter toggles the position of the empty space in a different dimension. Consider $K = \ker\varphi$, which is a normal subgroup. By our definition of $\varphi$, $K$ must correspond with sequences of moves which fix the empty space. By the extension of Lemma \ref{lem:fifteen_fix_even}, $K \cong A_7$, as it acts like $A_7$ on the configurations in $V_{fix}$. Given that $K$ is normal in $G$, the product, $K\langle R \rangle$, is a subgroup of $G$, and since $R \notin K$, we get that $\vert K\langle R \rangle \vert = 5,040$. This subgroup will be a key piece of the decomposition of $G$.

On the other hand, consider $Z$, the center of $G$. Computationally, we verified that $\vert Z \vert = 4$.\footnote{The nontrivial elements of $Z$ can be given by the sequences: $(RU)^2(RB)^2UB(RB)^2UBRB$,
$(RU)^2RB(RU)^2(BU)^2RURB$, and 
$(RU)^2BUBR(BU)^2BR(BU)^2$.}
The configurations produced by applying these elements to $\iota$ are shown in Figure~\ref{fig:center}. One can easily verify by inspection that the intersection of $K\langle R \rangle$ and $Z$ is trivial, as no element in $K \langle R \rangle$ will move the empty space far enough to reach the non-trivial configurations in Figure~\ref{fig:center}. Furthermore,
\begin{equation}
    \vert K\langle R \rangle Z \vert = \frac{\vert K \langle R \rangle \vert \cdot \vert Z \vert}{\vert K \langle R \rangle \cap Z \vert } = \frac{5040 \cdot 4}{1} = 20160 = \vert G \vert.
\end{equation}
Since $K\langle R \rangle Z \leqslant G$, then $K\langle R \rangle Z = G$. Therefore, since $K\langle R \rangle \cap Z$ is trivial and $Z$ commutes with $K\langle R \rangle$, we have that $G \cong K \langle R \rangle \times Z$. By determining the structure of these components, we will then obtain the full structure of $G$.

\begin{figure}[H]
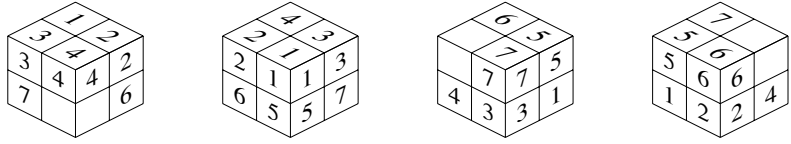

    \centering
    \scalebox{0.9}{
        \varikon{1}{2}{3}{4}{5}{6}{7}{}
        \varikon{4}{3}{2}{1}{}{7}{6}{5}
        \varikon{6}{5}{}{7}{2}{1}{4}{3}
        \varikon{7}{}{5}{6}{3}{4}{1}{2}
    }
    \caption{The configurations obtained by applying elements of the center, $Z$, to $\iota$. They correspond with $\iota$ itself, and $180^{\circ}$ rotations of the entire box, pivoting around the $U$, $B$, and $R$ axes.}
    \label{fig:center}
\end{figure}

\begin{lemma}\label{lem:K}
$K\langle R \rangle \cong S_7$, where $K$ is the kernel of the group homomorphism $\varphi$.
\end{lemma}
\begin{proof}
$K$ is normal in $K\langle R \rangle$ and $K \cap \langle R \rangle = \{e\}$, so $K\langle R \rangle = K \rtimes \langle R \rangle \cong A_7 \rtimes \mathbb{Z}_2$.\footnote{The structure of $K \rtimes \langle R \rangle$ is given by the automorphism $\phi_R$ of $K$ defined by $\phi_R(k) = RkR^{-1}$. That is, for $k_1, k_2 \in K$ and $r_1, r_2 \in \langle R \rangle$, multiplication is defined as $(k_1, r_1) \cdot (k_2, r_2) := (k_1 r_1 k_2 r_1^{-1}, r_1 r_2)$.} It is well known that $A_7 \rtimes \mathbb{Z}_2$ is isomorphic to either $A_7 \times \mathbb{Z}_2$ or $S_7$, so it remains to show the former is false. If it held, then $K\langle R \rangle$ would contain an element $kR$ of order $2$, with $k \in K$, commuting with all of $K\langle R \rangle$. Since $Z$ is the center, then $kR$ commutes with all of $G=K\langle R \rangle Z$, and thus $kR \in Z$. This is a contradiction, since $K\langle R \rangle \cap Z$ is trivial. 
This gives us $K \langle R \rangle \cong S_7$.

\end{proof}

\begin{lemma}\label{lem:Z}
The center $Z$ of the Varikon box group is a Klein four-group $(\mathbb{Z}_2)^2$.
\end{lemma}
\begin{proof}
Note by Figure~\ref{fig:center} that applying a $180^{\circ}$ rotation of the whole Varikon Box to any configuration maintains the numbers aligned along the $U$, $B$, or $R$ axes. From this it is clear that sequences yielding these configurations commute with every other sequence, and each non-trivial element has order $2$. It follows that $Z \cong (\mathbb{Z}_2)^2$.
\end{proof}

\begin{theorem}\label{thm:varikon_group}
The group of the Varikon box $G$ is isomorphic to $S_7 \times (\mathbb{Z}_2)^2$.
\end{theorem}
\begin{proof}
This follows trivially from Lemma~\ref{lem:K} and Lemma~\ref{lem:Z}.
\end{proof}

\section{The $A_5$ and $A_6$ Shortest Word Problem}
\label{sec:shortest_word}

One way to proceed toward a solution heuristic is by limiting the configurations to reduce the size of the problem. In practice, it is simple to locate the piece belonging in the position opposite of the empty space (the piece labelled $1$) and solve it. If we only consider the configurations in $V_{fix}$ where piece $1$ is in the solved position, then we could solve the remaining pieces with $\cycle[\;]{3,4,7}$, $\cycle[\;]{2,4,6}$, and $\cycle[\;]{5,6,7}$, by alternating between any two of $R,U,$ and $B$. With these restrictions, we can visualize the puzzle as just these $3$-cycles on pieces $2,3,4,5,6,$ and $7$, but for convenience, we will relabel these from $1$ to $6$, so that the permutations are now $\cycle[\;]{1,2,3}$, $\cycle[\;]{3,4,5}$ and $\cycle[\;]{5,6,1}$. Figure~\ref{fig:subproblem} shows a visual representation of this simplified puzzle.

One can easily verify that these 3-cycles generate $A_6$. This reduces the puzzle to a word problem: given a permutation in $A_6$, what is the shortest way to write it as a product of these 3-cycles and their inverses? We can solve this sub-problem by performing the inverse of this product. We can even reduce it further, by solving piece $6$ of the sub-problem (which is easy in practice). This leaves us with solving the remaining $5$ pieces using only the $3$-cycles $\cycle[\;]{1,2,3}$ and $\cycle[\;]{3,4,5}$, which are enough to generate $A_5$ (a visualization of this is shown in Figure~\ref{fig:subproblem}).

\begin{figure}[H]
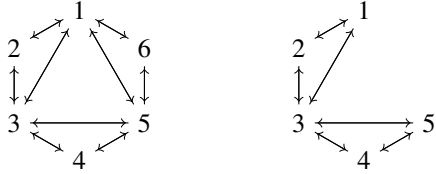

    \centering
    \cyclesix{1}{2}{3}{4}{5}{6}
    \hspace{1cm}
    \cyclefive{1}{2}{3}{4}{5}
    \caption{A visualization of the 2$\times$2$\times$2 Varikon Box sub-problems. The left assumes the piece opposite the empty space is in the solved position, and restricts to three 3-cycles. The right assumes two pieces are solved, and restricts to two 3-cycles.}
    \label{fig:subproblem}
\end{figure}

We computed the shortest-length product of permutations in $A_5$ in terms of $\cycle[\;]{1,2,3}$ and $\cycle[\;]{3,4,5}$. The worst-case found was the permutation $\cycle[\;]{1,2}\cycle[\;]{4,5}$, with a word-length of $6$. Furthermore, it was the only permutation to have this length:

\begin{equation}
    \cycle[\;]{1,2}\cycle[\;]{4,5} = \cycle[\;]{3,4,5}^{-1}\cycle[\;]{1,2,3}\cycle[\;]{3,4,5}\cycle[\;]{1,2,3}^{-1}\cycle[\;]{3,4,5}^{-1}\cycle[\;]{1,2,3}.
\end{equation}

We then computed the shortest-length product of permutations in $A_6$ in terms of the generators $\cycle[\;]{1,2,3}, \cycle[\;]{3,4,5},$ and $\cycle[\;]{5,6,1}$. We found $46$ permutations achieve the maximum word-length of $5$. Which corresponds with the permutation $\cycle[\;]{2,4,6}$:

\begin{equation}
    \cycle[\;]{2,4,6} = \cycle[\;]{1,2,3}\cycle[\;]{3,4,5}^{-1}\cycle[\;]{5,6,1}^{-1}\cycle[\;]{3,4,5}\cycle[\;]{1,2,3}^{-1}.
\end{equation}

\begin{figure}[H]
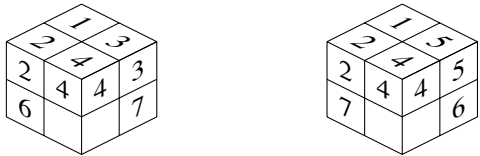

    \centering
    \varikon{1}{3}{2}{4}{5}{7}{6}{}
    \hspace{1cm}
    \varikon{1}{5}{2}{4}{3}{6}{7}{}
    \caption{Examples of the worst-cases for the two sub-problems. The left shows the $A_5$ sub-problem, which takes $24$ moves with this method. The right shows the $A_6$ sub-problem, which takes $20$ moves.}
    \label{fig:worst_subproblem} 
\end{figure}

These examples are shown in their puzzle-form in Figure~\ref{fig:worst_subproblem}. Setting up the $A_6$ sub-problem takes at most $2$ moves, since the orientation of the puzzle can be freely changed, and the worst-case scenario is when the piece labelled $1$ starts adjacent to the empty space. Combining this with the worst-case for the $A_6$ word sub-problem, it takes at most $22$ moves to solve the Varikon Box with this method, assuming one can solve the shortest word problem. This is comparable to the known worst-case which is $19$, found by a brute-force method~\cite{Scherphuis}.

\section{Conclusion and Future Work}
\label{sec:conclusion}
We have analyzed the 2$\times$2$\times$2 Varikon Box by describing moves of the puzzle as a group action on its configurations. The group associated with sequences of moves has an order equal to the number of reachable configurations, and is isomorphic to $S_7~\times~(\mathbb{Z}_2)^2$. Additionally, there exist larger versions of the Varikon Box (for example, 3$\times$3$\times$3 and 4$\times$4$\times$4). It remains to be seen if a similar analysis can be applied to an $n\times n \times n$ Varikon Box. More abstractly, one could consider higher-dimensional variants: that is, a group generated by $\{X_1,...,X_k\}$, where $X_i^n = e$ for each $i$, which could have further applications to discrete dynamical systems in general.

\section*{Acknowledgements}
This work was supported  by the Natural Sciences and Engineering Research Council of Canada (NSERC), funding ref.~RGPIN-2019-04669, and the University of Waterloo.

\bibliographystyle{amsplain}
\bibliography{referenc}
\end{document}